\documentclass[english,10pt]{amsart}
\usepackage[T1]{fontenc}
\usepackage[utf8]{inputenc}
\usepackage[english]{babel}
\usepackage{graphicx,xcolor}

\usepackage{amsmath,amsfonts,enumerate,amssymb}
\usepackage{amsthm}
\usepackage{float}
\usepackage[scr=boondox]{mathalpha}

\theoremstyle{plain}
\newtheorem{Lem}{Lemma}[section]

\newtheorem{Teo}{Theorem}[section]
\newtheorem{Cor}{Corollary}[section]
\newtheorem{NotOurTheorem}{Theorem}

\newtheorem{NotOurLemma}{Lemma}

\theoremstyle{definition}
\newtheorem{Def}{Definition}[section]
\newtheorem{Rem}{Remark}[section]

\renewcommand{\le}{\leqslant}
\renewcommand{\ge}{\geqslant}

\newcommand*{\xx}{\mathbf{x}}
\newcommand*{\yy}{\mathbf{y}}
\newcommand*{\rr}{\mathbf{r}}

\newcommand*{\RR}{\mathbb{R}}
\newcommand{\uRR}{\underline{\mathbb{R}}}
\newcommand*{\NN}{\mathbb{N}}

\newcommand*{\Simplex}[1]{\overline{S^{#1}}}
\newcommand*{\Sinf}{\Simplex{\infty}}

\newcommand*{\Reg}{R}
\newcommand*{\Dif}{D}

\newcommand*{\segmReg}{R^{[a,b]}}
\newcommand*{\segmDif}{D^{[a,b]}}
\newcommand*{\segmK}{K^{[a,b]}}
\newcommand*{\segmJ}{J^{[a,b]}}
\newcommand*{\segmF}{F^{[a,b]}}

\newcommand*{\cReg}{\mathscr{R}^{[0,1]}}
\newcommand*{\cDif}{\mathscr{D}^{[0,1]}}

\newcommand*{\cK}{\mathscr{K}^{[0,1]}}
\newcommand*{\cJ}{\mathscr{J}^{[0,1]}}

\DeclareMathOperator{\dist}{dist}

\def\ds{\displaystyle}

\begin{document}

\title{Homeomorphism theorem for sums of translates on the real axis}

\author{Tatiana M. Nikiforova}

\date{}


\thanks{The work was performed as part of research conducted
in the Ural Mathematical Center with the financial support
of the Ministry of Science and Higher Education of the Russian
Federation
(Agreement number 075-02-2025-1549).}

\begin{abstract}
In this paper, we study \emph{sums of translates} on the real axis. These functions generalize logarithms of weighted algebraic polynomials. Namely, we are dealing with the following functions
\[
F(\yy,t) := J(t) + \sum \limits_{j=1}^n K_j(t-y_j), \quad \yy := (y_1,\ldots,y_n), \ y_1 \le \ldots \le y_n,
\]
where the \emph{field function} $J$ is a function defined on $\RR$, which is "admissible" for the \emph{kernels} $K_1,\ldots,K_n$ concave on $(-\infty,0)$ and on $(0,\infty)$ and having a singularity at $0.$ We consider "local maxima"
\begin{gather*}
\begin{aligned}
m_0(\yy) & := \sup \limits_{t \in (-\infty, y_1]} F(\yy, t), \quad
m_n(\yy) := \sup \limits_{t \in [y_n, \infty)} F(\yy, t),\\
m_j(\yy) & := \sup \limits_{t \in [y_j, y_{j+1}]} F(\yy, t), \quad j = 1,\ldots,n-1, 
\end{aligned}
\end{gather*}
and the difference function
\[
\Dif(\yy) := (m_1(\yy)-m_0(\yy), m_2(\yy)-m_1(\yy),\ldots,m_n(\yy)-m_{n-1}(\yy)). 
\]
We prove that, under certain assumptions on monotonicity of the kernels, $\Dif$ is a homeomorphism between its domain and $\RR^n.$ 
\end{abstract}

\maketitle

\section{Introduction} 
In this paper, we study \emph{sums of translates} on the real axis. These functions generalize logarithms of weighted algebraic polynomials. Let $w \ge 0$ be a weight function and let $P(t) := (t-y_1) \cdot \ldots \cdot (t-y_n)$, $y_1 \le \ldots \le y_n.$
Consider
\[
\log (w(t) |P(t)|)= \log w(t) + \sum \limits_{j=1}^n \log |t-y_j|.
\]
Replacing $\log |\cdot-y_j|$ by $K_j(\cdot-y_j)$ and $\log w$ by $J$, we obtain \emph{the sum of translates function}
\[
F(\yy,t) := J(t) + \sum \limits_{j=1}^n K_j(t-y_j), \quad \yy := (y_1,\ldots,y_n), \ y_1 \le \ldots \le y_n, \ t \in \RR,
\]
where $J: \RR \to \uRR := \RR \cup \{-\infty\}$ is the \emph{field function} and the \emph{kernels} $K_1,\ldots,K_n$ are concave on $(-\infty,0)$ and on $(0,\infty)$ and have a singularity at $0.$ 

In our previous paper~\cite{MyMinimax}, we studied the minimax problem for the sums of translates on the real axis. The main method was a reduction to the minimax theorem for a segment proved by B.~Farkas, B.~Nagy and Sz.~Gy.~R\'{e}v\'{e}sz  (see \cite{FNRMinimax}, \cite{NewFNRMinimax}). The uniqueness of the minimax point on the real axis followed immediately from this reduction. To prove the uniqueness of the minimax point on the segment, Farkas, Nagy and R\'{e}v\'{e}sz used the so-called homeomorphism theorem \cite[Th.~7.1]{FNR}, which was also proved by them.
Now, we prove a similar homeomorphism theorem for the real axis. As for the segment, this result provides the uniqueness of the minimax point on the real axis, independently of the specific reduction technique.

\begin{Def}
Let $0 < p \le \infty.$
A function $K \colon (-p,0) \cup (0, p) \to \RR$ is called a \emph{kernel function} if $K$ is concave on $(-p,0)$ and on $(0, p)$
and $\lim \limits_{t \downarrow 0} K(t) = \lim \limits_{t \uparrow 0} K(t)$,  which are either real or equal to $-\infty$.

We say that the kernel $K$ is \emph{strictly concave} if $K$ is strictly concave on $(-p,0)$ and on $(0, p)$.

If $K$ is (strictly) decreasing on $(-p, 0)$ and (strictly) increasing on $(0, p)$, then $K$ is called (strictly) \emph{monotone}.

We extend $K$ by defining
\[
K(0) := \lim \limits_{t \to 0} K(t), \quad K(-p) := \lim \limits_{t \downarrow -p} K(t), \quad K(p) := \lim \limits_{t \uparrow p} K(t).
\]
If
$K(0) = -\infty$, the kernel function $K$ is called \emph{singular}.
\end{Def}

\begin{Def}
Let $A$ be a segment, a semiaxis or $\RR$. We call a function $J \colon A \to \uRR := \RR \cup \{-\infty\}$ an \emph{external $n$-field function} or simply a \emph{field} on $A$ if $J$ is bounded above on $A$ and it assumes finite values at more than $n$ different points of $A,$ where in the case of a segment we count boundary points with weights $1/2.$
\end{Def}

In the case of a segment, it is necessary that there are at least $n$ interior points and some additional one anywhere in the segment, where the field is finite. Therefore, we impose precisely such conditions on the weights of the points to ensure consistency with the case of a segment.

We consider "local maxima"
\begin{equation}
\label{def:localMaxima}
\begin{aligned}
m_0(\yy) & := \sup \limits_{t \in (-\infty, y_1]} F(\yy, t), \quad
m_n(\yy) := \sup \limits_{t \in [y_n, \infty)} F(\yy, t),\\
m_j(\yy) & := \sup \limits_{t \in [y_j, y_{j+1}]} F(\yy, t), \quad j = 1,\ldots,n-1, 
\end{aligned}
\end{equation}
and the difference function
\[
\Dif(\yy) := (m_1(\yy)-m_0(\yy), m_2(\yy)-m_1(\yy),\ldots,m_n(\yy)-m_{n-1}(\yy)). 
\]
We prove that, under certain assumptions on monotonicity of the kernels, $\Dif$ is a homeomorphism between its domain and $\RR^n.$

Results of this kind are inspired by the problem of optimizing the Lagrange interpolation of a continuous function. Let us give an overview of the results known to us. 

Let $f$ be a function continuous on $[0,1]$. Let us introduce the open simplex
\[
S^{[0,1]} := \{\yy := (y_1, \ldots, y_n) \in \RR^n: \ 0 =: y_0 < y_1 < \ldots < y_n < y_{n+1} := 1\}.
\]
Denote by $\pi_{n+1}$ the space of polynomials of degree at most $n+1$ and by $P_{\yy}: C_{[0,1]} \to  \pi_{n+1}$ the Lagrange interpolation operator
\[
(P_{\yy}f)(t) := \sum \limits_{j=0}^{n+1} f(y_j) \ell_j(t), \quad \ell_j(t) := \prod \limits_{i=0, \ i \neq j}^{n+1} \dfrac{t-y_i}{y_j-y_i}, \ j = 0, \ldots,n+1.
\]
It is well-known, see \cite[p.~88]{Rivlin}, that
\[
\|P_{\yy}f - f\|_{C_{[0,1]}} \le \dist(f, \pi_{n+1}) (1+\|P_{\yy}\|),
\]
where $\|P_{\yy}\|$ is the operator norm. Therefore, it is natural to minimize $\|P_{\yy}\|$ by $\yy$ to optimize the interpolation. It is known that $ \|P_{\yy}\|=\|\Lambda_{\yy}\|_{C_{[0,1]}},$ where
$\ds \Lambda_{\yy}(t) := \sum \limits_{j=0}^{n+1} |\ell_j(t)|$ is the Lebesgue function.
Thus, this optimization problem is reduced to the minimax problem for the functions $\Lambda_{\yy},$ i.e.,
\[
\min \limits_{\yy} \|P_{\yy}\| = \min \limits_{\yy} \max \limits_{t \in [0,1]} |\Lambda_{\yy} (t)|.
\]
Moreover, to obtain the maximum of $\Lambda_{\yy}$ on $[0, 1]$, it is sufficient to consider the local maxima $\lambda_j(\yy) := \max \limits_{t \in [y_j,y_{j+1}]} \Lambda_{\yy}(t), \ j =0,\ldots,n.$

In 1931, S.~N.~Bernstein \cite{Bernstein} conjectured that the minimum of $\|\Lambda_{\yy}\|_{C_{[0,1]}}$ is attained when $\yy$ is an \emph{equioscillation point}, i.~e.,
\[
\lambda_0(\yy)=\ldots=\lambda_n(\yy).
\]
In 1977, T.~A.~Kilgore \cite{KilgoreFull} proved Bernstein's conjecture. Moreover, he showed the uniqueness of the equioscillation point. More precisely, Kilgore’s note \cite{KilgoreNote} describing the proof of the statement "the minimax point is an equioscillation point" was first published, and a few months later, the complete proof was presented in \cite{KilgoreFull}. Subsequently, C.~R.~de~Boor and A.~Pinkus \cite{deBoorPinkus} also proved Bernstein's conjecture by a different method. Their approach is valuable, as their observation can be applied in various other contexts. In fact, they obtained the following general result.
\begin{NotOurTheorem}
The difference function
\[
\Dif_{\lambda}: S^{[0,1]} \to \RR^n, \quad \yy \mapsto (\lambda_1(\yy)-\lambda_0(\yy),\ldots,\lambda_n(\yy)-\lambda_{n-1}(\yy))
\]
is a homeomorphism between $S^{[0,1]}$ and $\RR^n.$
\end{NotOurTheorem}
It immediately implies that there is exactly one equioscillation point. De~Boor and Pinkus then refer to Kilgore's earlier note with the proof that the minimax point is an equioscillation point, and together all this proves Bernstein's conjecture.

In 1996, Y.~G.~Shi \cite{Shi} considered more general functions $\varphi_j(\cdot), \ j=0,\ldots,n,$ instead of $\lambda_j(\cdot).$ Shi supposed that all the functions $\varphi_j$ are continuously differentiable on $S^{[0,1]}$ and satisfy conditions
\[ 
\lim \limits_{\min \limits_{0 \le j \le n} (y_{j+1}-y_j) \to 0} \ \max \limits_{0 \le i \le n-1} |\varphi_{i+1}(\yy) - \varphi_i(\yy)| = \infty
\]
and
\[
\Phi_k(\yy) := \det \left( \dfrac{\partial \varphi_i(\yy)}{\partial y_j} \right)^{n \ n}_{j=1, \ i=0, \ i \neq k} \neq 0, \quad \yy \in S^{[0,1]}, \ k=0,\ldots,n.
\]
Shi dealt with the following minimax problem: find a vector $\yy = (y_1,\ldots,y_n)$ with $0 < y_1 < \ldots < y_n < 1,$ that minimizes 
$\max \limits_{j=0,\ldots,n} \varphi_j(\cdot).$ 
Under these assumptions, Shi proved that there exists a unique extremal point and it has the equioscillation property. Moreover, Shi obtained the homeomorphism theorem for the difference function
\[
\Dif_{\varphi}: S^{[0,1]} \to \RR^n, \quad \yy \mapsto (\varphi_1(\yy)-\varphi_0(\yy),\ldots,\varphi_n(\yy)-\varphi_{n-1}(\yy)).
\]
In particular, this theorem implies the uniqueness of the minimax point. 

The sums of translates and the minimax problem for such functions were first considered by P.~C.~Fenton in 2000 \cite{Fenton}. He considered one kernel with assumptions of monotonicity, smoothness, singularity of its derivative at $0$, and a concave field $J$ continuous at the ends of the segment.
Fenton's original goal was to prove P.~D.~Barry's conjecture from 1962 on the growth of entire functions, which he succeeded in 1981 \cite{FentonBarry}. More precisely, A.~A.~Goldberg proved this conjecture a little earlier \cite{Goldberg}, but Fenton obtained other interesting results in the theory of entire functions using his approach \cite{FentonOther1}, \cite{FentonOther2}.

In 2018, B.~Farkas, B.~Nagy and Sz.~Gy.~R\'{e}v\'{e}sz presented a solution of the minimax problem and a homeomorphism result for sums of translates $\ds F(\yy,t) := K_0(t) + \sum \limits_{j=1}^n K_j(t-y_j)$ on a torus \cite{TorusFNR}. Here $K_0, \ldots, K_n: \ \RR \to [-\infty, 0)$ are $2\pi$-periodic functions, strictly concave on $(0, 2\pi)$. Assuming that for each $j=0,\ldots,n$ the function $K_j$ belongs to $C^2(0,2\pi)$ with $K''_j<0$ and $K_j(0)=K_j(2\pi)=-\infty,$ they proved that the difference function of local maxima is a homeomorphism between its domain and $\RR^n$. On the one hand, the sums of translates approach is more specific than Shi's. On the other hand, the authors provided an example \cite[Ex.~5.13]{TorusFNR} demonstrating that Shi's result is not applicable in their settings.

In 2021, Farkas, Nagy and R\'{e}v\'{e}sz proved a homeomorphism theorem for sums of translates on the segment \cite[Th.~7.1]{FNR}.
The minimax problem for the sums of translates on the segment was also deeply studied by Farkas, Nagy and R\'{e}v\'{e}sz in \cite{FNRMinimax}, \cite{NewFNRMinimax}. In particular, they consider positive numbers $r_1,\ldots,r_n$, a kernel $\segmK$ and sums of translates of the following form
\[
F(\yy,t) = J(t) + \sum \limits_{j=1}^n r_j \segmK(t-y_j), \quad \yy \in \Simplex{[a,b]}, \ t \in [a,b],
\]
where
\[
\Simplex{[a,b]} := 
\{(y_1,\ldots,y_n) \in \RR^n: \ a \le y_1 \le \ldots \le y_n \le b\}.
\]

They proved that if the kernel $\segmK$ is monotone, singular and strictly concave, then there exists a minimax point characterized by the equioscillation property. To prove the uniqueness of the equioscillation point, Farkas, Nagy and R\'{e}v\'{e}sz apply the homeomorphism theorem.
Their research inspired the author to obtain similar results for the minimax problem on the real axis \cite{MyMinimax} and to write this paper.

Consider the sums of translates
\[
F(\yy,t) := J(t) + \sum \limits_{j=1}^n K_j(t-y_j), \quad \yy \in \Simplex{[a,b]}, \ t \in [a,b].
\]
Denote
\begin{align*}
m^{[a,b]}_0(\yy) &:= \sup \limits_{t \in [a, y_1]} F(\yy, t), \quad
m^{[a,b]}_n(\yy) := \sup \limits_{t \in [y_n, b]} F(\yy, t),\\
m^{[a,b]}_j(\yy) &:= \sup \limits_{t \in [y_j, y_{j+1}]} F(\yy, t), \quad j = 1,\ldots,n-1.
\end{align*}
Let us introduce \emph{the regularity set}
\[
\segmReg := \{\yy \in \Simplex{[a,b]}: \ m^{[a,b]}_j(\yy) \neq -\infty \text{ for } j = 0,\ldots,n \}
\]
and \emph{the difference function} $\segmDif: \segmReg \to \RR^n$
\[
\segmDif(\yy) := (m^{[a,b]}_1(\yy)-m^{[a,b]}_0(\yy), m^{[a,b]}_2(\yy)-m^{[a,b]}_1(\yy),\ldots,m^{[a,b]}_n(\yy)-m^{[a,b]}_{n-1}(\yy)).
\]
If $b = -a$, then we write $\Simplex{b}, \ m^{b}_j(\yy), \ \Reg^b.$

As we mentioned above, the following homeomorphism theorem was proven by Farkas, Nagy and R\'{e}v\'{e}sz \cite[Th.~7.1]{FNR}.
\begin{NotOurTheorem} \label{theorem:FNR}
Let $a < b,$ the kernel functions $\segmK_1,\ldots,\segmK_n: (a-b,0) \cup (0,b-a) \to \RR$ be singular, strictly concave, and $\segmJ: [a,b] \to \uRR$ be an $n$-field function. Assume that 
\begin{gather} \label{theorem:FNR:kernel}
\left(\segmK_j(t) - \segmK_j(t-(b-a))\right)' \ge 0
\quad \text{for almost all} \quad 
t \in (0,b-a), \quad j = 1,\ldots,n, 
\end{gather}
and
\begin{gather} \label{theorem:FNR:field}
\segmJ(a) = \lim \limits_{t \downarrow a} \segmJ(t) = -\infty 
\quad \text{or} \quad
\segmJ(b) = \lim \limits_{t \uparrow b} \segmJ(t) = -\infty.
\end{gather}
Then the difference function
\[
\segmDif : \segmReg \to \RR^n,\quad \yy \mapsto (m^{[a,b]}_1(\yy)-m^{[a,b]}_0(\yy),m^{[a,b]}_2(\yy)-m^{[a,b]}_1(\yy),\ldots,m^{[a,b]}_n(\yy)-m^{[a,b]}_{n-1}(\yy))
\]
is a homeomorphism between $\segmReg$ and $\RR^n$. Moreover, $\segmDif$ is locally bi-Lipschitz.
\end{NotOurTheorem}
\begin{Rem} \label{reductionToZeroOne}   
The authors proved the homeomorphism theorem for $a=0, \ b=1.$ 
For convenience of further application, we formulate this theorem for an arbitrary segment.
 The theorem on $[0,1]$ can be trivially extended to $[a,b]$ by applying the linear transformation
\[
\chi(t) := \dfrac{t-a}{b-a}
\]
that maps the segment $[a,b]$ onto $[0,1]$. Let us show that the homeomorphism theorem proven on $[0,1]$ implies Theorem~\ref{theorem:FNR}.

Assume that kernels $\segmK_j, \ j=1,\ldots,n,$ and a field $\segmJ$ satisfy the conditions of Theorem~\ref{theorem:FNR}. 
Let 
\begin{align*}
\cK_j(x) & := \segmK_j((b-a)x), \quad x \in [-1,1], \ j=1,\ldots,n, \\
\cJ(x) & := \segmJ(a + (b-a)x), \quad x \in [0,1].
\end{align*}
We have for $t, c \in [a,b]$ that
\begin{equation} \label{reductionToZeroOne:identity}
\begin{aligned}
\segmK_j(t-c) & = \cK_j(\chi(t) - \chi(c)), \quad j=1,\ldots,n,\\
\segmJ(t) & = \cJ(\chi(t)).
\end{aligned}
\end{equation}
Hence, taking into account the conditions on $\segmK_j$ and $\segmJ$, we have that the kernels $\cK_j$ are strictly concave and singular and $\cJ$ is an $n$-field function satisfying
\[
\cJ(0) = \lim \limits_{x \downarrow 0} \cJ(x) = -\infty
\quad \text{or} \quad
\cJ(1) = \lim \limits_{x \uparrow 1} \cJ(x) = -\infty.
\] 
Let us show that
\begin{gather} \label{reductionToZeroOne:kernel}
\left(\cK_j(x) - \cK_j(x - 1) \right)'_x \ge 0 \quad \text{for almost all} \quad 
x \in (0,1).
\end{gather}
We have
\begin{align*}
\left(\segmK_j(t) - \segmK_j(t-(b-a))\right)'_t &= \left(\cK_j(\chi(t) - \chi(0)) - \cK_j(\chi(t) - \chi(b-a)) \right)'_t \\
&= \left(\cK_j(t/(b-a)) - \cK_j(t/(b-a) - 1) \right)'_t.
\end{align*}
For $t \in (0, b-a)$, substituting $x := t/(b-a) \in (0,1)$ and using \eqref{theorem:FNR:kernel}, we obtain \eqref{reductionToZeroOne:kernel}.

Therefore, the difference function $\cDif$ is a homeomorphism between the regularity set $\cReg$ and $\RR^n.$ 
Denote $\boldsymbol \chi(\yy) := (\chi(y_1),\ldots,\chi(y_n))$. Using \eqref{reductionToZeroOne:identity}, it is easy to see that
\[
\boldsymbol \chi(\segmReg) = \cReg, \quad \segmDif(\yy) \equiv \cDif(\boldsymbol \chi(\yy)), \quad \yy \in \segmReg.
\]
So, we obtain that $\segmDif$ is a homeomorphism, too.
\end{Rem}

In addition to \cite[Th.~7.1]{FNR}, the authors obtained homeomorphism theorems with other conditions on the kernels and field. Conditions \eqref{theorem:FNR:field} of Theorem \ref{theorem:FNR} may be replaced by so-called cusp conditions at the ends of the segment \cite[Th.~7.5]{FNR}. Moreover, if the derivatives of the kernel differences in \eqref{theorem:FNR:kernel} are bounded below by some $c > 0,$ then the field can be arbitrary \cite[Th.~2.1]{FNR}.

Our goal is to prove an analogue of Theorem \ref{theorem:FNR} for sums of translates on $\RR$. Our method relies on reducing the problem to Theorem \ref{theorem:FNR}, using an approach developed primarily in our previous paper \cite{MyMinimax}. Let us introduce the main definitions and state our result.

By the concavity of a kernel $K$, the set where $K'$ is defined has full measure. In the following, we consider $K'$ and the limits of $K'$ on this set. When $K$ is defined on $(-\infty,0) \cup (0,\infty),$ the following condition, called \emph{generalized monotonicity}, is important to us:
\begin{gather} \label{def:GM} \tag{$GM$}
\lim \limits_{t \to -\infty} K'(t) \le \lim \limits_{t \to \infty} K'(t).
\end{gather}

\begin{Rem} \label{rem:finiteLimitsGM}
Note that if $\eqref{def:GM}$ holds, then $\lim \limits_{t \to -\infty} K'(t)$ and $\lim \limits_{t \to \infty} K'(t)$ are finite. Indeed, since $K'$ is non-increasing at all points of its domain, the limits in $\eqref{def:GM}$ exist in the extended sense (taking values in the extended real line), and $-\infty < \lim \limits_{t \to -\infty} K'(t), \ \lim \limits_{t \to \infty} K'(t) < \infty.$ Now the finiteness of these limits immediately follows from the property \eqref{def:GM}.
\end{Rem}

\begin{Rem}
Obviously, if $K$ is defined on $(-\infty,0) \cup (0, \infty)$ and monotone, then it satisfies \eqref{def:GM}.

On the other hand, condition \eqref{def:GM} for a kernel function $K$ is equivalent to the existence of a number $c$ such that the kernel $K(t) - ct$ is monotone.
Indeed, suppose that \eqref{def:GM} holds. Consider the kernel $\widetilde{K}(t) := K(t) - ct$, where $c := \lim \limits_{t \to \infty} K'(t)$. We get
\[
\lim \limits_{t \to \infty} \widetilde{K'}(t) = \lim \limits_{t \to \infty} K'(t) - c = 0.
\]
Comparing this with property \eqref{def:GM}, one can also conclude that
\[
\lim \limits_{t \to -\infty} \widetilde{K'}(t) = \lim \limits_{t \to -\infty} K'(t) - c \le \lim \limits_{t \to \infty} K'(t) - c = 0.
\]
Due to the concavity of $K$, the function $\widetilde{K'}$ is non-increasing at all points of its domain, therefore $\widetilde{K'}(t) \le 0$ almost everywhere on $(-\infty,0)$, and hence $\widetilde{K}$ is non-increasing on $(-\infty,0)$. Similarly, $\lim \limits_{t \to \infty} \widetilde{K'}(t) = 0$ implies that $\widetilde{K'} \ge 0$ a.e. on $(0,\infty)$, that guarantees that $\widetilde{K}$ is non-decreasing on $(0,\infty)$. Consequently, the kernel $\widetilde{K}$ is monotone.

Conversely, if there exists a number $c$ such that the kernel $K(t) - ct$ is monotone, then its derivative satisfies
\[
\lim \limits_{t \to -\infty} K'(t) - c \le 0 \quad \text{and} \quad \lim \limits_{t \to \infty} K'(t) - c \ge 0.
\]
Hence, obviously, \eqref{def:GM} holds.
\end{Rem}

Let us introduce the "infinite closed simplex"
\[
\Sinf := \{(y_1, \ldots, y_n) \in \RR^n \colon -\infty < y_1 \le y_2 \le \ldots \le y_n < \infty\}.
\]
Similarly to the finite interval case, we can consider the regularity set
\[
R := \{\yy \in \Sinf: \ m_j(\yy) \neq -\infty \text{ for } j = 0,\ldots,n \}
\]
and also the corresponding difference function 
\begin{gather} \label{def:differenceFunc}
D: \Reg \to \RR^n,\quad \yy \mapsto (m_1(\yy)-m_0(\yy),m_2(\yy)-m_1(\yy),\ldots,m_n(\yy)-m_{n-1}(\yy)).    
\end{gather}

\begin{Def}
Let $K_j: (-\infty, 0) \cup (0, \infty) \to \RR, \ j = 1,\ldots,n,$ be kernels. An $n$-field function $J$ defined on $\RR$ 
is said to be \emph{admissible} (for $K_1,\ldots,K_n$) if 
\[
\lim_{|t| \to \infty} \left( J(t) + \sum \limits_{j=1}^n K_j(t) \right) = -\infty.
\]
\end{Def}
This condition is a version of the condition for admissible weights in weighted potential theory, see \cite[p.~26]{Saff-Totik}

The main result of this paper is the following.
\begin{Teo} \label{theorem:main}
Suppose that the singular, strictly concave kernel functions $K_j: (-\infty, 0) \cup (0, \infty) \to \RR, \ j = 1,\ldots,n,$ satisfy \eqref{def:GM} and $J: \RR \to \uRR$ is an admissible $n$-field function for $K_1, \ldots,K_n$. Then the difference function defined in \eqref{def:differenceFunc} 
is a homeomorphism between $\Reg$ and $\RR^n$. Moreover, $\Dif$ is locally bi-Lipschitz.
\end{Teo}
The proof of Theorem~\ref{theorem:main} is organized as follows. In Sections~2 and 3, we prove auxiliary results concerning kernels and fields. Section~4 is devoted to showing that $D$ is a local homeomorphism, while in Section~5 we prove that $D$ is a proper map. Finally, in Section~6, we apply Ho's theorem, which states that a proper local homeomorphism is global.

\section{Simple lemmas about kernels and fields}
In this section, we need the following extension of the fundamental theorem of calculus for concave functions. It is known \cite[p.~9]{RobertsVarberg} that if $g$ is a concave function on an interval~$I$, then for~any~$[a,b]~\subset~I$
\begin{gather} \label{NewtonLeibnitz}
g(b) - g(a) = \int \limits_{[a,b]} g'(t) dt.
\end{gather}

\begin{Lem} \label{lemma:simpleShifts}
Let $K: (-\infty, 0) \cup (0, \infty) \to \RR$ be a kernel function. 
\begin{enumerate}[label=(\alph*)]
\item If $0 < t_1 < t_2 < t_2+h$ or $t_1 < t_2 < t_2 + h < 0$, then
\begin{gather} \label{lemma:simpleShifts:eq1}
K(t_2+h) - K(t_1+h) \le K(t_2) - K(t_1).
\end{gather}
Moreover, $K'$ is non-increasing at all points of its domain.
\item If $K$ satisfies \eqref{def:GM} and $t_1 <  t_1+h < 0  < t_2$, then 
\begin{gather} \label{lemma:simpleShifts:eq2} 
K(t_2) - K(t_1) \le K(t_2+h) - K(t_1+h).
\end{gather}
In particular, for almost all $t_1, t_2$ such that $t_1 < 0  < t_2$ we have
\begin{gather} \label{lemma:simpleShifts:eq2:derivative}
K'(t_1) \le K'(t_2).
\end{gather}
\end{enumerate}
\end{Lem}

\begin{proof}
\begin{enumerate}[label=(\alph*)]
\item Inequality \eqref{lemma:simpleShifts:eq1} follows from the definition of concavity. Its proof can be found, e.~g., in \cite[Lemma 10]{Rankin}.

To prove the statement about $K'$, it is sufficient to assume that $K'$ is defined at $t_2$ and $t_2+h,$ divide \eqref{lemma:simpleShifts:eq1} by $t_2-t_1$ and pass to the limit.

\item Sufficiently using \eqref{NewtonLeibnitz}, Part (a) and the property \eqref{def:GM} for $K$, we get
\begin{gather*}
(K(t_2+h) - K(t_2)) - (K(t_1+h) - K(t_1)) = 
\\ \int \limits_{[t_2,t_2+h]} K'(t) - \int \limits_{[t_1,t_1+h]} K'(t) \ge
h \left(\lim \limits_{t \to \infty} K'(t) - \lim \limits_{t \to -\infty} K'(t)\right) \ge 0.
\end{gather*}
By Remark \ref{rem:finiteLimitsGM}, the limits in \eqref{def:GM} are finite, so their difference is well-defined. Thus we have obtained \eqref{lemma:simpleShifts:eq2}.

To get \eqref{lemma:simpleShifts:eq2:derivative}, one can separate the terms in \eqref{lemma:simpleShifts:eq2} with respect to $t_1$ and $t_2$, placing them on the left- and right-hand sides respectively, then divide by $h$ and take the limit $h \to 0$.
\end{enumerate}
\end{proof}

Let us show that admissibility of a field is equivalent to the following more general property.
\begin{Lem} \label{lemma:admissible}
Let $K_j \colon (-\infty, 0) \cup (0, \infty) \to \RR, \ j = 1,\ldots,n,$ be kernels.
If a field $J$ is admissible for $K_1,\ldots,K_n$, then for any $(y_1,\ldots,y_n)\in \RR^n$ 
\[
\lim \limits_{|t| \to \infty} \left( J(t) + \sum \limits_{j=1}^n K_j(t-y_j) \right) = -\infty.
\]
\end{Lem}

\begin{proof}
Let us prove the lemma for $t \to \infty$. If $t \to -\infty$, then we can consider $\widetilde{J}(t) := J(-t)$ and $\widetilde{K}_j(t) := K_j(-t), \ j = 1,\ldots,n,$ and apply what is proved for $t \to \infty$.

1. Let $y_i := \max \limits_{1 \le j \le n} \{ y_j \}$. Suppose that all $K'_j$ be bounded below for large arguments.
By our assumption and Part (a) of Lemma \ref{lemma:simpleShifts}, there are $C, \ L > 0$ such that
\begin{gather} \label{lemma:admissible:boundedBelow}
|K'_j(t)| \le C, \quad t \ge L, \ j=1,\ldots,n.
\end{gather}

We consider $t \ge L$ such that $\min \{t, t-y_i\} \ge L$.
Using \eqref{NewtonLeibnitz} and \eqref{lemma:admissible:boundedBelow}, we get
\[
K_j(t-y_j)-K_j(t) \le C |y_j|, \quad j=1,\ldots,n.
\]
Therefore, by the admissibility of $J$, we obtain
\[
J(t) + \sum \limits_{j=1}^n K_j(t-y_j) \le J(t) + \sum \limits_{j=1}^n (K_j(t) + C |y_j|) \to -\infty, \quad t \to \infty.
\]

2. Now assume that $K'_s$ is not bounded below for some $s \in \{1,\ldots,n\}$ and for~large~$t.$
Denote 
\[
f(\yy,t) := \sum \limits_{j=1}^n K_j(t-y_j).
\]
Note that $f(\yy,\cdot)$ is concave on $(y_i, \infty)$. Take arbitrary point $t_0 \in (y_i, \infty)$ where $f'_t$ exists. For any $t \in (y_i, \infty)$, we have \cite[p. 12, Th. D]{RobertsVarberg}
\begin{gather} \label{lemma:admissible:tangentLine}
f(\yy,t) \le f'_t(\yy,t_0) \cdot (t-t_0) + f(\yy,t_0).
\end{gather}
By Part (a) of Lemma \ref{lemma:simpleShifts} for $K'_j, \ j=1\ldots,n,$ and our assumption regarding $K_s$, we conclude that $f'_t(\yy,t_0) < 0$ for large $t_0$ as a sum of non-increasing functions and a function not bounded below.
Due to \eqref{lemma:admissible:tangentLine}, we obtain
\[
\lim \limits_{t \to \infty} f(\yy,t) = -\infty.
\]
Since $J$ is bounded above, we finally get
\[
F(\yy,t) = J(t) + f(\yy,t) \to -\infty, \quad t \to \infty.
\] 
\end{proof}

\section{Behavior of sums of translates for large arguments}
In this section, we prove a key lemma which allows us to reduce our problem on the axis to the case of the segment. Another form of this statement was proven for sums of translates with positive multiples of a single kernel in \cite[Lemma~4.3]{MyMinimax}. Now we are dealing with several kernels and we need a slightly modified estimation for kernels, so we will provide a proof.

The following statement is well-known, see, e.g., \cite[Lemma 4.1]{MyMinimax}.
\begin{NotOurLemma}
\label{lemma:uniformContin}
Suppose that a function $g$ is concave, non-decreasing on the semiaxis $[M,\infty)$ and is continuous at $M.$ Then $g$ is uniformly continuous on $[M,\infty)$.
\end{NotOurLemma} 

We also need the following lemma. 
\begin{Lem} \label{lemma:-infty}
If $K_j \colon (-\infty, 0) \cup (0, \infty) \to \RR, \ j=1,\ldots,n,$ are kernels and $J \colon \RR \to \uRR$ is an admissible field, then for any $\yy \in \Sinf$
we have
\[
\lim \limits_{\xx \to \yy, \ |t| \to \infty}
F(\xx,t) = -\infty.
\]
\end{Lem}

\begin{proof}
Let us prove the lemma for $t \to \infty.$ For $t \to -\infty$ the proof is carried out by considering the reflection $\widetilde{K}(t) := K(-t), \ \widetilde{J}(t) = J(-t).$

Denote $I_1 := \{j: \ K_j \text{ is non-decreasing on } (0, \infty)\}.$ Let $j \in I_1.$ 
By Lemma \ref{lemma:uniformContin}, the function $K_j$ is uniformly continuous on $[1,\infty)$. Hence if $\|\xx - \yy\|$ is sufficiently small, then there exists $C_j > 0$ such that for large $t$ we have $K_j(t-x_j) - K_j(t-y_j) \le C_j$.

Now consider $I_2 := \{j: \ K_j \text{ is not non-decreasing on } (0, \infty)\}$ and $j \in I_2.$ By the concavity, $K_j$ decreases for large $t$. Since $\xx$ converges, there is $c$ such that $x_j \le c$ for all $j.$ So, for large $t$ we have that $K_j(t-x_j) \le K_j(t-c).$

Using the inequalities above, we get for large $t$ and for $\xx$ sufficiently close to $\yy$
\[
F(\xx,t) = J(t) + \sum \limits_{j =1}^n K_j(t-x_j) \le
J(t) + \sum \limits_{j \in I_1} (K_j(t-y_j) + C_j)+ \sum \limits_{j \in I_2} K_j(t-c).
\]
Applying Lemma \ref{lemma:admissible}, together with the fact that for all $j \in J_2, \ \lim \limits_{t \to \infty} K_j(t-c)$ is finite or $-\infty$, i.e., it is strictly less that $\infty,$ we obtain $\lim \limits_{\xx \to \yy, \ t \to \infty} F(\xx,t) = -\infty.$ 
\end{proof}

Now, let us prove the main lemma of this section.
\begin{Lem} \label{lemma:tau} Let $K_j \colon (-\infty, 0) \cup (0, \infty) \to \RR, \ j = 1,\ldots,n,$ be kernels and $J: \RR \to \uRR$ be an admissible field. 
Then for any $N \in \NN$ there is a number $\tau_N > N$ such that for each $ \yy \in \Simplex{N}$
\[
\left( t < -\tau_N \implies F(\yy,t) \le m_0(\yy) - 1 \right) 
\ \text{and} \
\left( t > \tau_N \implies F(\yy,t) \le m_n(\yy) - 1 \right).
\]
In particular, this implies that
\begin{gather} \label{lemma:tau:maxima}
m^{\tau_N}_j(\yy) = m_j(\yy), \quad \yy \in \Simplex{N}, \quad j = 0, \ldots, n.
\end{gather}
\end{Lem}

\begin{proof}
Let us prove the statement for $t < -\tau_N$. The proof of the second part is similar.

Assume for a contradiction that for some $N \in \NN$
\[
\forall M \in \NN, \ M > N \quad \exists \,\yy^M \in \Simplex{N} \quad \exists\, t^M\in \RR \quad \left(t^M < -M \quad \& \quad F(\yy^M,t^M) > m_0(\yy^M) - 1 \right).
\]
We have the bounded sequence $\{\yy^M\} \subset \Simplex{N}.$ By the Bolzano-Weierstrass theorem, there is a convergent subsequence $\{\yy^{M_k}\}.$
Set
\[
\lim \limits_{k \to \infty} \yy^{M_k} =: \yy^{*} = (y^*_{1},\ldots,y^*_{n}) \in \Simplex{N}.
\]
Take $t \in (-\infty, -N)$. By our assumption, we have for all $k$ with $-M_k < t < -N$
\[
J(t) = F(\yy^{M_k},t) - \sum \limits_{j=1}^n K_j(t - y^{M_k}_j) \le F(\yy^{M_k},t^{M_k}) + 1 - \sum \limits_{j=1}^n K_j(t - y^{M_k}_j).
\]
Using continuity of $K_j$ at $t - y^*_j < 0$ and Lemma \ref{lemma:-infty}, we get
\begin{align*}
J(t) & \le \lim \limits_{k \to \infty} \left( F(\yy^{M_k},t^{M_k}) + 1 - \sum \limits_{j=1}^n K_j(t - y^{M_k}_j) \right) \\
& = \lim \limits_{k \to \infty} F(\yy^{M_k},t^{M_k}) + 1 - \sum \limits_{j=1}^n K_j(t - y^*_j) = -\infty.
\end{align*}
So, $J(t) \equiv -\infty$ for $t < -N.$ We have a contradiction with our assumption, since then also $F(\yy^M, t^M) = -\infty.$
\end{proof}

\section{Local homeomorphism}
In this section, we prove a statement about local homeomorphism by reducing the problem to a segment. For $N \in \NN$ take $\tau_N > N$ from Lemma~\ref{lemma:tau}.
Consider
\[
K^{\tau_N}_j := K_j \vert_{[-2\tau_N,2\tau_N]}, \quad J^{\tau_N} := J \vert_{[-\tau_N,\tau_N]}.
\]

Obviously, if $K_j$ is a singular (strictly) concave kernel function, then $K^{\tau_N}_j$ has the same properties. And it is also clear that if $J$ is an $n$-field function on $\RR$, then $J^{\tau_N}$ is an $n$-field on $[-\tau_N,\tau_N]$ for large $N$.

Note that by \eqref{lemma:tau:maxima},
\begin{gather}
\label{differenceFunc:identity}
\Dif^{\tau_N}(\yy) = \Dif(\yy), \quad \yy \in \Reg^N,
\end{gather}
where $R^N$ is the regularity set for $F \vert_{[-N,N]}.$

Farkas, Nagy and R\'{e}v\'{e}sz in \cite{FNR} established the results mentioned in this section for sums of translates on the segment $[0,1]$. By applying the reasoning from Remark~\ref{reductionToZeroOne}, these results can be extended to the segment $[a,b]$.

\begin{Lem} \label{lemma:connectedness}
Let $K_j \colon (-\infty, 0) \cup (0, \infty) \to \RR, \ j = 1,\ldots,n,$ be singular kernels and $J: \RR \to \uRR$ be an admissible field. Then the regularity set $\Reg$ is open and pathwise connected.
\end{Lem}

\begin{proof}
Note that
\[
R = \bigcup \limits_{N \in \NN} R^N.
\]
Farkas, Nagy and R\'{e}v\'{e}sz proved in \cite[Prop.~4.1]{FNR} that the sets $R^N$ are open and pathwise connected.
Hence $R$ is also open and pathwise connected as a union of sets with these properties.
\end{proof}

We will show that an analogue of the following property \cite[Lemma~7.3]{FNR} for $\segmDif$ holds for $\Dif$.

\begin{NotOurLemma} \label{lemma:localHomSegm}
Let $\segmK_j: (a-b,0) \cup (0,b-a) \to \underline{\RR}, \ j = 1,\ldots,n,$ be singular kernels and $\segmJ: [a,b] \to \uRR$ be an $n$-field function.
Assume that for $j \in \{1,\ldots,n\}$
\[
\left(\segmK_j(t)-\segmK_j(t-(b-a))\right)' \ge 0 \quad \text{almost everywhere on $(0,b-a)$},
\]
and for any $\yy \in \segmReg$ there exists $\eta > 0$ such that
\begin{equation}
\label{localHomSegm:nearEndPoints}
\begin{aligned} 
\text{either} \ \segmF(\yy,t) & \le m^{[a,b]}_0(\yy)-1, \ t \in [a,a+\eta] \\
\text{or} \ \segmF(\yy,t) & \le m^{[a,b]}_n(\yy)-1, \ t \in [b-\eta,b].    
\end{aligned}
\end{equation}
Then the difference function $\segmDif$ is a local homeomorphism. Moreover, $\segmDif$ is locally bi-Lipschitz.
\end{NotOurLemma}

In fact, Farkas, Nagy and R\'{e}v\'{e}sz assume conditions \eqref{theorem:FNR:field} instead of \eqref{localHomSegm:nearEndPoints}.
In \cite[Lemma~7.3]{FNR} it is shown that \eqref{theorem:FNR:field} implies conditions \eqref{localHomSegm:nearEndPoints}, used in the proof of the local homeomorphism.
For $\RR$, conditions \eqref{localHomSegm:nearEndPoints} in the problem on $[-\tau_N,\tau_N]$ are consequences of Lemma~\ref{lemma:tau}. Note that the main role in the proof of Lemma~\ref{lemma:tau} is played by admissibility of the field, similar to conditions \eqref{theorem:FNR:field}.

\begin{Lem} \label{lemma:localHom}
Let $K_j \colon (-\infty, 0) \cup (0, \infty) \to \RR, \ j = 1,\ldots,n,$ be singular kernels satisfying \eqref{def:GM} and $J: \RR \to \uRR$ be an admissible field.
\begin{enumerate}
\item The difference function $\Dif$ is a local homeomorphism. Moreover, $\Dif$ is locally bi-Lipschitz.
\item The functions $m_j: \Sinf \to \uRR, \ j=0,\ldots,n,$ are continuous in the extended sense, i.e., as functions with values on the extended real axis.
\end{enumerate}
\end{Lem}

\begin{proof}
\begin{enumerate}
\item 
Let us apply Lemma \ref{lemma:localHomSegm} with $[a,b] = [-\tau_N, \tau_N]$ and
\[
K^{\tau_N}_j := K_j \vert_{[-2\tau_N,2\tau_N]}, \quad 
J^{\tau_N} := J \vert_{[-\tau_N,\tau_N]}, \quad 
F^{\tau_N} := F \vert_{[-\tau_N,\tau_N]}
\]
for large $N$.
As discussed above, $K^{\tau_N}_j$ are singular kernels on $(-2\tau_N,2\tau_N)$, and $J^{\tau_N}$ is an $n$-field on $[-\tau_N,\tau_N]$. 

Without loss of generality, we can assume that the statement of Lemma \ref{lemma:tau} is true for $\tau_{N}-1$. Then for $\yy \in \Simplex{N} \cap \Reg^N$ and $m_i^{\tau_N}(\yy) = m_i(\yy),$ for $i=0,\ldots,n,$ we get
\[
F^{\tau_N}(\yy,t) \le m^{\tau_N}_0(\yy)-1, \ t \in [-\tau_N,-\tau_N+1]
\]
and
\[
F^{\tau_N}(\yy,t) \le m^{\tau_N}_n(\yy)-1, \ t \in [\tau_N-1, \tau_N].
\]
By \eqref{lemma:simpleShifts:eq2:derivative} for $j = 1,\ldots,n$ we have
\[
\left(K^{\tau_N}_j(t) - K^{\tau_N}_j(t-2\tau_N)\right)'_t \ge  0 \quad \text{almost everywhere on $(0,2\tau_N)$}.
\]
So, all conditions of Lemma \ref{lemma:localHomSegm} are satisfied. We obtain that the difference function $\Dif^{\tau_N}$ is a local homeomorphism, and it is locally bi-Lipschitz on $\Reg^N$. 
By \eqref{differenceFunc:identity}, $\Dif$ also has these properties on $\Reg^N$ and thus on $\Reg,$ since $N$ can be arbitrarily large.
\item In \cite[Lemma~3.3]{FNR} it is proven that the functions $m^{\tau_N}_j, \ j=0,\ldots,n,$ are continuous on $\Simplex{N}$ in the extended sense.
By \eqref{lemma:tau:maxima}, we conclude that $m_j, \ j=0,\ldots,n,$ are extended continuous on $\Simplex{N}$, and hence on $\Sinf$ as well due to the arbitrariness of $N$.
\end{enumerate}
\end{proof}

\section{Properness of the difference function}
The idea of the proof of the following lemma was communicated to us by Szil\'{a}rd Gy. R\'{e}v\'{e}sz. We already used this idea in \cite[Th.~3.1]{MyMinimax} with his permission. However, we now prove a statement of a different setting, and therefore we need a slightly different estimate in the proof. Moreover, we are dealing with several kernels satisfying \eqref{def:GM} instead of just one monotone kernel.

\begin{Lem} \label{lemma:MRS}
Assume that the kernels $K_j \colon (-\infty, 0) \cup (0, \infty) \to \RR, \ j = 1,\ldots,n,$ satisfy \eqref{def:GM}, and $J$ is an admissible field. Let $\{\yy^N\} \subset \Reg$ be an unbounded sequence convergent in the extended sense. Then there exists $i \in \{1, \ldots, n\}$ such that
\[
|m_{i}(\yy^N) - m_{i-1}(\yy^N)| \to \infty, \quad N \to \infty.
\]
\end{Lem}

\begin{proof}
1. Denote $\yy^N := (y^N_1, \ldots, y^N_n), \ y^N_0 := -\infty, \ y^N_{n+1} := \infty.$ Since $\yy^N$ is convergent in the extended sense, then $y^N_n \to \infty$, or $y^N_1 \to -\infty$. Without loss of generality, we may assume the first case occurs, since the proof of the second case is similar. Let $i := \min \{j: y^N_j \to \infty\}.$

2. Let us show that there exists $z > 0$ such that $J(z) \neq -\infty$ and $z \in (y^N_{i-1}, y^N_{i})$ for $N$ large enough.

By definition of $i$, we have either $y^N_{i-1} \to -\infty$ or $y^N_{i-1} \to y_{i-1} \in \RR.$ In both cases, there exist $A > 0, \ N_0 \in \NN$ such that $y^N_{i-1} < A-1$ for all $N > N_0.$ Since $\{\yy^N\} \subset \Reg$ for all $N$, there are arbitrarily large points where $J$ is finite. Fix one such point $z > A.$ Then, for all $N > N_0$, we have $y^N_{i-1} < z - 1$. Moreover, since $y^N_i \to \infty$ as $N \to \infty$, there exists $N_1 > N_0$ such that $y^N_i > z+1$ for all $N > N_1$. For $N > N_1,$ we conclude that $z \in (y^N_{i-1},y^N_i)$ and $|z-y^N_j| > 1$ for all $j.$ Next we consider these $N$.

3. Let us estimate
\[
F(\yy^N,t) = J(t) +
\sum \limits_{j < i} K_j(t-y^N_j) + 
\sum \limits_{j \ge i} K_j(t-y^N_j)
\]
for $t \in (y^N_i, y^N_{i+1})$.

If $j < i,$ apply \eqref{lemma:simpleShifts:eq1} with $t_1 = 1, \ t_2 = t-z+1, \ h = z-1-y^N_j$. 
If $j > i$, use \eqref{lemma:simpleShifts:eq2} with $t_1 = z-y^N_j, \ t_2 = 1, \ h = t - z.$ In both cases, we obtain
\begin{equation} \label{lemma:MRS:eq1}
K_j(t-y^N_j) - K_j(z-y^N_j) \le K_j(t-z+1) - K_j(1), \quad j \neq i.
\end{equation}

Since $K_i$ satisfies \eqref{def:GM}, by \eqref{lemma:simpleShifts:eq2} with $t_1 = z-y^N_i, \ t_2 = t-y^N_i, \ h = y^N_i-z-1$, we get
\begin{equation} \label{lemma:MRS:eq3}
K_i(t-y^N_i) - K_i(z-y^N_i) \le K_i(t-z-1) - K_i(-1).
\end{equation}

By \eqref{lemma:MRS:eq1} and \eqref{lemma:MRS:eq3},
we obtain
\begin{align*}
F(\yy^N,t) & \le
J(t) + \sum \limits_{j \neq i} (K_j(z-y^N_j) + K_j(t-z+1)  - K_j(1)) \\
& +
K_i(z-y^N_i) + K_i(t-z-1) - K_i(-1) \\  
& = F(\yy^N,z) + J(t) - J(z) 
+ \sum \limits_{j \neq i} (K_j(t-z+1) - K_j(1)) \\
& + K_i(t-z-1) - K_i(-1). 
\end{align*}
Let $\ds C := - J(z) - \sum \limits_{j \neq i} K_j(1) - K_i(-1).$ By the above estimate, we have
\begin{equation} \label{lemma:MRS:eq5}
F(\yy^N,t) \le F(\yy^N,z) + J(t) + \sum \limits_{j \neq i} K_j(t-z+1)  
+ K_i(t-z-1) + C.
\end{equation}

Since $J$ is admissible, by Lemma \ref{lemma:admissible}, we have that 
\begin{equation}
\label{lemma:MRS:eq6} 
J(t) + \sum \limits_{j \neq i} K_j(t-z+1)  
+ K_i(t-z-1) \to -\infty, \quad t \to \infty.     
\end{equation}

Take arbitrary $L > 0$. Since $t \ge y^N_{i} \to \infty,$ for large $N$ we obtain that \eqref{lemma:MRS:eq6} is less than $-L-C$.
Hence, by \eqref{lemma:MRS:eq5}, for large $N$, we get that for $t \in (y^N_i, y^N_{i+1})$
\[
F(\yy^N,t) \le F(\yy^N,z) - L.
\]
Therefore, taking into account the choice of $z$, we have
\[
m_{i}(\yy^N) \le F(\yy^N,z) - L \le m_{i-1}(\yy^N) - L.
\]
Thus
\[
m_{i}(\yy^N) - m_{i-1}(\yy^N) \to -\infty, \quad N \to \infty.
\]
\end{proof}

\begin{Def}
A function $g: A \to B$ between two Hausdorff topological spaces is called \emph{proper} if for any compact set $Q \subset B$ we have that $g^{-1}(Q)$ is also a compact set \cite[p.~20]{Bredon}.
\end{Def}

\begin{Lem} \label{lemma:proper}
Assume that the kernel functions $K_j \colon (-\infty, 0) \cup (0, \infty) \to \RR, \ j = 1,\ldots,n,$ are singular and satisfy \eqref{def:GM}. Let a field function $J$ be admissible. Then
the difference function $\Dif: \Reg \to \RR^n$ is proper.
\end{Lem}

\begin{proof}
Let $Q \subset \RR^n$ be a compact set. We need to show that $\Dif^{-1}(Q)$ is also a compact set, i.~e., $\Dif^{-1}(Q)$ is closed and bounded.

Let us prove that $\Dif^{-1}(Q)$ is closed. 
By Lemma \ref{lemma:localHom}, $\Dif$ is continuous. Hence, since $Q$ is closed, $\Dif^{-1}(Q)$ is relatively closed in $\Reg$, i.~e.,
\[
\Dif^{-1}(Q) = A \cap \Reg, \quad \text{where $A$ is closed in $\RR^n$}.
\]
Consider $\{\yy^N\} \subset \Dif^{-1}(Q)$ and let $\yy^N \to \yy, \ N \to \infty.$ 
We have that $\yy \in A,$ since $\{\yy^N\} \subset A$ and $A$ is closed in $\RR^n$. 
If we prove that $\yy \in \Reg$, it immediately follows that $\Dif^{-1}(Q)$ is closed. 

The field $J$ is finite at least at $n+1$ points, so 
there is $i \in \{0, \ldots, n\}$ such that $m_{i}(\yy) \in \RR.$ Moreover, since $\{\yy^N\} \subset \Dif^{-1}(Q),$ there exists $C > 0$ such that for all $N \in \NN$
\[
|m_j(\yy^N)-m_{j-1}(\yy^N)| \le C, \quad j=1,\ldots,n.
\]
By Lemma~\ref{lemma:localHom}, the functions $m_j$ are continuous. 
Therefore, the differences $m_j(\yy)-m_{j-1}(\yy)$ also satisfy $|m_j(\yy)-m_{j-1}(\yy)| \le C$ for all $j=1,\ldots,n$.
Taking into account the finiteness of $m_i(\yy)$, this implies that $\yy \in \Reg.$ 

It remains to show that $\Dif^{-1}(Q)$ is bounded. Suppose, contrary to our claim, that there exists an unbounded sequence $\{\yy^N\} \subset \Dif^{-1}(Q)$. Without loss of generality, let $\yy^N$ converges in the extended sense. By Lemma \ref{lemma:MRS}, 
there is $i \in \{1,\ldots,n\}$ such that
\[
|m_{i}(\yy^N)-m_{i-1}(\yy^N)| \to \infty, \quad N \to \infty.
\] 
Hence $Q \supset \Dif(\{\yy^N\})$ is not bounded. This contradicts the assumption that $Q$ is a compact set, and this finishes the proof.
\end{proof}

\section{Proof of the main result}
To prove Theorem \ref{theorem:main}, we need the following sufficient condition, due to C.~W.~Ho \cite{Ho}, that a local homeomorphism is a global one.
\begin{NotOurTheorem} \label{theorem:Ho}
Let $A, B$ be pathwise connected, Hausdorff
topological spaces with $B$ simply connected. Let $f: A \to B$ be a proper local homeomorphism. Then $f$ is a global homeomorphism between $A$ and $B$.
\end{NotOurTheorem}

\begin{proof}[Proof of Theorem \ref{theorem:main}]
By Lemma \ref{lemma:connectedness}, the regular set $\Reg$ is pathwise connected.
The difference function
\[
\Dif: \Reg \to \RR^n, \ \yy \mapsto (m_1(\yy)-m_0(\yy), m_2(\yy)-m_1(\yy),\ldots,m_n(\yy)-m_{n-1}(\yy))
\]
is a local homeomorphism by Lemma \ref{lemma:localHom} and it is proper by Lemma \ref{lemma:proper}. Therefore, by Theorem \ref{theorem:Ho}, we obtain that $\Dif$ is a global homeomorphism.
\end{proof}

\section{Homeomorphism theorem for sums of translates \\ on the semiaxis}
Let $J^+$ be an $n$-field on $[0, \infty)$, $K_j: (-\infty,0) \cup (0,\infty) \to \RR, \ j=1,\ldots,n,$ be kernels. Admissibility of the field $J^+$ for $K_1,\ldots,K_n$ can be defined as follows:
\[
\lim_{t \to \infty} \left( J^+(t) + \sum \limits_{j=1}^n K_j(t) \right) = -\infty.
\]
Denote
\[
\Simplex{[0, \infty)} := \{(y_1, \ldots, y_n) \in \RR^n: \ 0 \le y_1 \le \ldots \le y_n < \infty \}.
\]
Consider sums of translates on $[0, \infty)$
\[
F^+(\yy,t) := J^+(t) + \sum \limits_{j = 1}^n K_j(t-y_j), \quad \yy \in \Simplex{[0, \infty)}, \quad t \in [0, \infty),
\]
and local maxima
\begin{align*}
m^+_0(\yy) & := \sup \limits_{t \in [0, y_1]} F^+(\yy, t), \quad
m^+_n(\yy) := \sup \limits_{t \in [y_n, \infty)} F^+(\yy, t),\\
m^+_j(\yy) & := \sup \limits_{t \in [y_j, y_{j+1}]} F^+(\yy, t), \quad j = 1,\ldots,n-1.
\end{align*}
Let $R^+$ be the regular set in this setting, i.e.,
\[
R^+ := \{\yy \in \Simplex{[0, \infty)}: \ m^+_j(\yy) \neq -\infty \text{ for } j = 0,\ldots,n \}.
\]

\begin{Cor}
Suppose that the singular, strictly concave kernel functions $K_j: (-\infty, 0) \cup (0, \infty) \to \RR, \ j = 1,\ldots,n,$ satisfy \eqref{def:GM} and $J^+: [0,\infty) \to \uRR$ is an admissible $n$-field function for $K_1,\ldots,K_n$.
Then the difference function
\[
\Dif^+: \Reg^+ \to \RR^n,\quad \yy \mapsto (m^+_1(\yy)-m^+_0(\yy),m^+_2(\yy)-m^+_1(\yy),\ldots,m^+_n(\yy)-m^+_{n-1}(\yy))
\]
is a homeomorphism between $\Reg^+$ and $\RR^n$. Moreover, $\Dif^+$ is locally bi-Lipschitz.
\end{Cor}

\begin{proof}
Consider
\[
F(\yy,t)= J(t) + \sum \limits_{j = 1}^n K_j(t-y_j), \quad \yy \in \Simplex{[0,\infty)}, \quad t \in \RR,
\]
where 
\[
J(t) = 
\begin{cases}
-\infty, & t < 0,\\
J^+(t), & t \ge 0.
\end{cases}
\]
Note that $K_j, \ j=1,\ldots,n,$ and $J$ satisfy the conditions of Theorem \ref{theorem:main}. Therefore, the difference function $\Dif: \Reg \to \RR^n$ for $F$ is a homeomorphism, and it is locally bi-Lipschitz. Moreover, the definition of $J$ implies that
\[
R^+ = R, 
\quad \text{and} \quad 
m_j^+(\yy) = m_j(\yy), \ \yy \in R^+.
\]
Hence $\Dif^+ \equiv \Dif,$ and $\Dif^+$ is a homeomorphism, and it is locally bi-Lipschitz, too.
\end{proof}

\section{Homeomorphism theorem for weighted generalized \\ polynomials on the real axis}
Let $r_1, \ldots, r_n > 0$ be arbitrary. Denote $\mathbf{r} := (r_1,\ldots,r_n)$. Consider the following set of monic generalized nonnegative polynomials \cite[p.~392]{BorweinErdelyi} of degree $r~:=~\sum \limits_{j=1}^n r_j$  
\[
\mathcal{P}_\rr(\RR) := \left\{p(\yy,t) = \prod \limits_{j=1}^n |t-y_j|^{r_j}: \ -\infty < y_1 < \ldots < y_n < \infty \right\}.
\]
Let $\RR^n_+$ denote the subset of vectors from $\RR^n$ with positive coordinates.

\begin{Cor}
Let $w: \RR \to [0, \infty)$ be a bounded above function assuming non-zero values at more than $n$ points. Assume that
\[
\lim \limits_{|t| \to \infty} w(t) \cdot t^{r} = 0.
\]
Let 
\[
S^{\infty} := \{(y_1, \ldots, y_n) \in \RR^n: \ -\infty < y_1 < \ldots < y_n < \infty\}, \quad 
\]
and 
\[
X := \{t \in \RR: \ w(t) = 0\}.
\]
For convenience, let $y_0 := -\infty, \ y_{n+1} := \infty.$
Denote 
\[
\Reg := \{\yy \in S^{\infty}: \ (y_j, y_{j+1}) \not \subseteq X 
\quad \text{for} \quad 
j=0,\ldots,n\}.
\]
Then the mapping
\[
\Reg \ni \yy \mapsto
\left(
\dfrac{\sup \limits_{t \in (y_1,y_{2})} w(t) p(\yy,t)}
{\sup \limits_{t \in (-\infty,y_1)} w(t) p(\yy,t)},
\ldots,
\dfrac{\sup \limits_{t \in (y_n,\infty)} w(t) p(\yy,t)}
{\sup \limits_{t \in (y_{n-1},y_n)} w(t) p(\yy,t)}
\right) \in \RR^n_+
\]
is a homeomorphism between $R$ and $\RR^n_+$. Moreover, it is locally bi-Lipschitz.
\end{Cor}
\begin{proof}
Let $K_j := r_j \log|\cdot|, \ j=1,\ldots,n, \ J := \log w.$ Obviously, these kernels and this field satisfy the conditions of Theorem \ref{theorem:main}. Therefore, the difference function
\[
\Dif: R \to \RR^n, \quad \yy \mapsto \left(
\sup \limits_{t \in (y_j,y_{j+1})} \log(w(t) p(\yy,t)) - \sup \limits_{t \in (y_{j-1},y_j)} \log(w(t) p(\yy,t))
\right)_{j=1}^n
\]
is a homeomorphism between $R$ and $\RR^n$, and it is locally bi-Lipschitz.
Since the logarithm strictly increases on $(0,\infty),$
\[
\sup \limits_{t \in (y_j,y_{j+1})} \log (w(t) p(\yy,t)) = \log \left( \sup \limits_{t \in (y_j,y_{j+1})} w(t) p(\yy,t) \right), \quad \yy \in \Reg.
\]
Let us write the difference function $D$ as
\[
\yy \mapsto \left(
\log \dfrac{\sup \limits_{t \in (y_j,y_{j+1})} w(t) p(\yy,t)}{\sup \limits_{t \in (y_{j-1},y_j)} w(t) p(\yy,t)}
\right)_{j=1}^n.
\]
Note that the function
\[
E: \ \RR^n \to \RR^n_+ \quad \xx \mapsto (\exp(x_1),\ldots,\exp(x_n))
\]
is a homeomorphism, and it is locally Lipschitz. It remains to see that the mapping in our statement is equal to $E \circ D.$
\end{proof}

\section{Interpolation by products of log-concave functions \\ with weight}
The homeomorphism theorem allows one to obtain results on interpolation by weighted products of log-concave functions. This is discussed in detail in \cite[Sect.~9]{FNR} for the homeomorphism theorem on the segment. The authors proved general results \cite[Th.~9.2,~9.6]{FNR} on the existence and uniqueness of an interpolation function. They also studied applications of these results to trigonometric interpolation \cite[Th.~9.7]{FNR}. Moreover, in \cite[Subsect.~9.3,~9.4]{FNR} the authors investigated the applications to moving node Hermite–Fej\'{e}r interpolation. 
Below we prove some analogues of the most general of these results on the real axis.

\subsection{Abstract log-concave interpolation on the real axis}
\begin{Teo} \label{theorem:interpolation}
Let $L_1,\ldots,L_n: \RR \to [0,\infty)$ be log-concave
functions vanishing at $0$ and satisfying 
\[
\lim \limits_{t \to -\infty} (\log L_j(t))' \le \lim \limits_{t \to \infty} (\log L_j(t))', \quad j=1,\ldots,n.
\]
Let $w: \RR \to [0, \infty)$ be a bounded above function assuming non-zero values at more than $n$ points. Assume that
\[
\lim \limits_{|t| \to \infty} w(t) \prod \limits_{j=1}^n L_j(t) = 0.
\]
For any $-\infty < x_0 < \ldots < x_n < \infty$ with $w(x_j) > 0$ and $\alpha_0, \ldots, \alpha_n > 0$ there are a unique
$C > 0$ and points $y_1 < y_2 < \ldots < y_n$ with $x_j < y_{j+1} < x_{j+1}$ for
each $j \in \{0,\ldots,n-1\}$ such that for the function
\[
G(t) := C w(t) \prod \limits_{j=1}^n L_j(t-y_j)
\]
we have
\[
G(x_j) = \alpha_j, \quad j = 0,\ldots,n.
\]
\end{Teo}

\begin{proof}
Let 
$K_j(t)  := \log L_j(t)$, $j=1,\ldots,n$,
\begin{equation}\label{9J}
J(t)  :=
\begin{cases}
\log w(t), \quad  t \in \{x_0,\ldots,x_n\}, \\
-\infty, \quad t \in \RR \setminus \{x_0,\ldots,x_n\}.
\end{cases}    
\end{equation}
It is easy to see that $K_j$ and $J$ satisfy the conditions of Theorem \ref{theorem:main}.

We claim that the regularity set has the following form
\begin{gather} \label{theorem:interpolation:Reg}
\Reg = \{\yy \in \Sinf: \ x_j < y_{j+1} < x_{j+1} \text{ for } j = 0,\ldots,n-1\}.
\end{gather}
Indeed, let $y_0 := -\infty$, $y_n := \infty$, and $j = 0,\ldots,n$. By the singularity of the kernels, $m_j(\yy)>-\infty$ if and only if $J(t_j) > -\infty$ for some point $t_j \in (y_j,y_{j+1}).$ Thus, there are $n+1$ intervals $(y_j,y_{j+1})$, and, by \eqref{9J}, there are exactly $n+1$ points 
$x_j$, where $J$ is finite. Consequently, $\Reg$ is of the form \eqref{theorem:interpolation:Reg}.

Take arbitrary $\alpha_0, \ldots, \alpha_n > 0.$ By Theorem \ref{theorem:main}, there is a unique $\yy \in \Reg$
such that
\begin{gather} \label{theorem:interpolation:eq1}
m_j(\yy)-m_{j-1}(\yy)=
\log(\alpha_{j}/\alpha_{j-1}), \quad j=1,\ldots,n.    
\end{gather}
We have
\begin{gather}
\label{theorem:interpolation:eq2}
m_j(\yy) = F(\yy,x_j) = \log w(x_j) + \sum \limits_{k=1}^n \log L_k(x_j-y_k), \quad j=0,\ldots,n.    
\end{gather}
Therefore, using \eqref{theorem:interpolation:eq1}, after exponentiating we get
\[
\dfrac{\exp m_j(\yy)}{\exp m_{j-1}(\yy)}
=
\dfrac{w(x_j) \prod \limits_{k=1}^n L_k(x_j-y_k)}
{w(x_{j-1}) \prod \limits_{k=1}^n L_k(x_{j-1}-y_k)} = \dfrac{\alpha_j}{\alpha_{j-1}}, \quad j=1,\ldots,n.  
\]
Hence
\[
\dfrac{w(x_j) \prod \limits_{k=1}^n L_k(x_j-y_k)}
{w(x_0) \prod \limits_{k=1}^n L_k(x_0-y_k)} = \dfrac{\alpha_j}{\alpha_0} , \quad j=1,\ldots,n,
\]
and for the function
\[
G(t):= Cw(t) \prod \limits_{k=1}^n L_k(t-y_k) 
\quad \text{with} \quad
C:=\dfrac{\alpha_0}{w(x_0) \prod \limits_{k=1}^n L_k(x_0-y_k)},
\]
we obtain
\[
G(x_j) = \alpha_j \quad j=0,\ldots,n.
\]

Let us show the uniqueness. Assume that for some $\widetilde{C}>0$ and $\widetilde{\yy} = (\widetilde{y}_1,\ldots,\widetilde{y}_n) \in \Reg$ we have
\[
\widetilde{C} w(x_j) \prod \limits_{k=1}^n L_k(x_j-\widetilde{y}_k) = \alpha_j \quad j=0,\ldots,n.
\]
Using \eqref{theorem:interpolation:eq2}, we have
\[
\log \widetilde{C} + m_j(\widetilde{\yy}) = 
\log \widetilde{C} +  \log w(x_j) + \sum \limits_{k=1}^n \log L_k(x_j-y_k) = \log \alpha_j \quad j=0,\ldots,n.
\]
Hence
\[
m_j(\widetilde{\yy})-m_{j-1}(\widetilde{\yy})=
\log(\alpha_{j}/\alpha_{j-1}), \quad j=1,\ldots,n.
\]
By Theorem \ref{theorem:main}, we conclude that $\widetilde{\yy}=\yy,$ therefore, $\widetilde{C} = C,$ too.
\end{proof}

\subsection{Moving node Hermite–Fej\'{e}r interpolation}
\begin{Teo}
Let $L_1,\ldots,L_n: \RR \to [0,\infty)$ be strictly log-concave
functions vanishing at $0$ and satisfying 
\[
\lim \limits_{t \to -\infty} (\log L_j(t))' \le \lim \limits_{t \to \infty} (\log L_j(t))', \quad j=1,\ldots,n.
\]
Let $w: \RR \to [0, \infty)$ be an upper semicontinuous function assuming non-zero values at more than $n$ points. Assume that
\[
\lim \limits_{|t| \to \infty} w(t) \prod \limits_{j=1}^n L_j(t) = 0.
\]
For convenience, let $y_0 := -\infty$ and $y_{n+1} := \infty$. Then, for any $\alpha_0, \ldots, \alpha_n > 0$ there are a unique $C > 0$ and points $y_1 < \ldots < y_n$ such that for the function
\[
G(t) := C w(t) \prod \limits_{j=1}^n L_j(t-y_j)
\]
there are $z_0, \ldots , z_n$ with $z_0 < y_1 < z_1 < y_2 < \ldots < z_{n-1} < y_n < z_n$
and
\[
G(z_j) = \alpha_j, \quad j = 0,\ldots,n,
\]
where $z_j$ is the maximum point of $G$ between $y_j$ and $y_{j+1}$ for each $j=0,\ldots,n.$
\end{Teo}

\begin{proof}
Let 
\[
K_j(t) := \log L_j(t), 
\quad j=1,\ldots,n, \quad J(t) := \log w(t).
\]
By Theorem \ref{theorem:main}, there is a unique $\yy \in \Reg$ such that 
\[
m_j(\yy)-m_{j-1}(\yy)=
\log(\alpha_{j}/\alpha_{j-1}), \quad j=1,\ldots,n.    
\]
Note that since $K_j$ are strictly concave and $w$ is upper semicontinuous, there are $z_0, \ldots , z_n$ with $z_0 < y_1 < z_1 < y_2 < \ldots < z_{n-1} < y_n < z_n$ such that
\[
F(\yy, z_i) = m_i(\yy),
\quad i=0,\ldots,n.
\]
The further proof is similar to the proof of Theorem \ref{theorem:interpolation} with $x_j$ replaced by $z_j, \\ j=0,\ldots,n$.
\end{proof}

\section*{Acknowledgements}
The author is grateful to Prof. Sz.~Gy.~R\'{e}v\'{e}sz for proposing the problem, useful references, discussions and detailed feedback during the preparation of the article.

The author is thankful to P.~Yu.~Glazyrina for continuous support, discussions on the proofs, and helpful suggestions regarding the structure of the article.

\bigskip
\bigskip
\bigskip
\noindent\parindent0pt
\hspace*{5mm}
\begin{minipage}{\textwidth}
\noindent
\hspace*{-5mm}Tatiana M. Nikiforova\\
Krasovskii Institute of Mathematics and Mechanics, \\
Ural Branch of the Russian Academy of Sciences\\
620990 Ekaterinburg, Russia,\\
Ural Federal University\\
620002 Ekaterinburg, Russia
\end{minipage}

\end{document}